\documentclass[11pt,reqno]{amsart}
\usepackage[utf8]{inputenc}
\usepackage[american]{babel}
\usepackage{amsmath,amsthm,amssymb,amsfonts,mathrsfs,mathtools,enumerate}
\usepackage{geometry}
\usepackage{xcolor}
\usepackage[colorlinks=true, linkcolor=magenta, citecolor=green, urlcolor=blue]{hyperref}

\newtheorem{theorem}{Theorem}[section]

\newtheorem{corollary}[theorem]{Corollary}
\newtheorem{lemma}[theorem]{Lemma}
\newtheorem{problem}[theorem]{Problem}
\newtheorem{proposition}[theorem]{Proposition}
\theoremstyle{definition}
\newtheorem*{remark}{Remark}

\newtheorem{definition}[theorem]{Definition}

\newcommand{\Aut}{\mathrm{Aut}}
\newcommand{\Bbox}{\,\Box\,}
\newcommand{\Cay}{\mathrm{Cay}}
\newcommand{\Sym}{\mathrm{Sym}}

\title{On the automorphism group of direct product of digraphs}

\author[G. Rao]{Guang Rao}
\address[Rao Guang]{College of Education Sciences, The Hong Kong University of Science and Technology (Guangzhou), 511453, P.R.China}
\email{guangrao@hkust-gz.edu.cn}

\author[Y.~Wang]{Yu Wang}
\address[Yu Wang]{Center for Combinatorics and LPMC, Nankai University, Tianjin, 300071, P.R.China}
\email{wangyu97@mail.nankai.edu.cn}

\author[B.~Xia]{Binzhou Xia}
\address[Binzhou Xia]{School of Mathematics and Statistics, The University of Melbourne, Parkville, VIC 3010, Australia}
\email{binzhoux@unimelb.edu.au}

\date{}

\begin{document}

\begin{abstract}
Determining the conditions under which the direct product of graphs $G$ and $H$ satisfies $\mathrm{Aut}(G\times H)=\mathrm{Aut}(G)\times\mathrm{Aut}(H)$ has been a problem of considerable interest since Sabidussi's classic work in the 1950s. We call such a pair $(G,H)$ stable, and unstable otherwise. Although much progress has been made for graph pairs, the general digraph case has remained completely open. In this paper, we initiate the study of the stability of digraph pairs, and then focus on the stability of a single digraph $G$. This is defined as the stability of the pair $(G,K_2)$ and has been studied extensively when $G$ is undirected. We establish a necessary and sufficient condition for a connected digraph to be unstable, and use it to derive four sufficient conditions for circulant digraphs to be unstable. Moreover, we prove the nonexistence of nontrivially unstable finite arc-transitive circulant digraphs and nontrivially unstable Cayley digraphs of abelian groups of odd order.


\end{abstract}

\maketitle

\section{Introduction}

Since Sabidussi's classic work in the 1950s, there has been considerable interest in the automorphism groups of various graph and digraph products, such as lexicographic product~\cite{Hahn1980,Sabidussi1959-2}, Cartesian product~\cite{Feigenbaum1986,Sabidussi1959-1} and strong product~\cite{BIKP2019,DI1970}.
Unlike these products, the automorphism group of the direct product of graphs is not well understood, although it has received substantial attention.
Even for the question of which graphs $G$ and $H$ satisfy $\Aut(G\times H)=\Aut(G)\times\Aut(H)$, a complete answer remains difficult to obtain (see the paragraph after Theorem~\ref{Thm-1} for partial results).
In this paper we initiate the study of the same question on digraphs.

A \emph{digraph} is a pair $(V,A)$ with \emph{vertex set} $V$ and \emph{arc set} $A$, where $A$ is a subset of $V\times V$.
For a digraph $G$, denote by $V(G)$ and $A(G)$ the vertex set and arc set of $G$ respectively, and denote $u\rightarrow_G v$ (or $u\rightarrow v$ if $G$ is clear from the context) if $(u,v)\in A(G)$.
The \emph{underlying graph} $\underline{G}$ of a digraph $G$ is the pair $(V(G),E)$, where the \emph{edge set} $E\coloneqq\{\{u,v\}\mid(u,v)\in A(G)\}$. A digraph $G$ is \emph{undirected} if $v\to_G u$ whenever $u\to_G v$, and we naturally identify undirected digraphs with their underlying graphs.
Note that digraphs and graphs in this paper may have loops (hence the edge set of a graph may contain singletons).

It is evident that $\Aut(G\times H)\geq\Aut(G)\times\Aut(H)$, where $\Aut(G)\times\Aut(H)$ acts on $V(G\times H)=V(G)\times V(H)$ in product action. We call a pair $(G,H)$ of digraphs \emph{stable} if $\Aut(G\times H)=\Aut(G)\times\Aut(H)$, and \emph{unstable} otherwise.

For a digraph $G$ and a vertex $v$ of $G$, the out-neighborhood and in-neighborhood of $v$ in $G$ are denoted by $G^+(v)$ and $G^-(v)$ respectively; that is,
\[
G^+(v)=\{w\in V(G)\mid(v,w)\in A(G)\}\ \text{ and }\ G^-(v)=\{u\in V(G)\mid(u,v)\in A(G)\}.
\]
A digraph $G$ is said to be \emph{twin-free} if there are no distinct vertices $u$ and $v$ of $G$ such that $G^+(u)=G^+(v)$ and $G^-(u)=G^-(v)$.

A digraph is called \emph{prime} if it has at least two vertices and cannot be represented as the direct product of two digraphs of fewer vertices.
Digraphs $G$ and $H$ are said to be \emph{coprime} if there is no digraph $K$ with at least two vertices such that $G\cong G_0\times K$ and $H\cong H_0\times K$ for some digraphs $G_0$, $H_0$ and $K$.
Recall that a digraph $G$ is said to be \emph{connected} if $\underline{G}$ is connected, and \emph{bipartite} if $\underline{G}$ is bipartite. By \emph{connected components} of a digraph $G$, we refer to the connected components of $\underline{G}$.

\begin{theorem}\label{Thm-1}
Let $(G,H)$ be a stable pair of digraphs such that $|V(G)|>1$ and $|V(H)|>1$. Then the following statements hold:
\begin{enumerate}[{\rm(a)}]
\item\label{Thm-1-a} $G$ and $H$ are coprime;
\item\label{Thm-1-b} $G$ and $H$ are both core-free;
\item\label{Thm-1-c} if neither $\Aut(G)$ nor $\Aut(H)$ is trivial, then both $G$ and $H$ are connected;
\item\label{Thm-1-d} if neither $\Aut(G)$ nor $\Aut(H)$ is trivial, then at least one of $G$ or $H$ is non-bipartite.
\end{enumerate}
\end{theorem}

Determining the stability of a general pair $(G,H)$ of digraphs is difficult, even if $G$ and $H$ are undirected; see~\cite{GLX2025,QXZZ2021} for general discussions, along with~\cite{Dorfler1974},~\cite[Theorem~8.18]{HIK2011} and~\cite[\S5]{Morris2021} for some sufficient conditions.
A meaningful and instructive case, which we will focus on, is when $H=K_2$.
We call a digraph $G$ \emph{stable} if $(G,K_2)$ is stable, and \emph{unstable} otherwise.
The stability of graphs has been studied extensively in the literature (see~\cite{FH2022,LMS2015,MSZ1992,Morris2021,NS1996,QXZ2021,Surowski2001,Surowski2003,Wilson2008}).
As a foundation for the study of digraph stability, we establish the following necessary and sufficient condition for a digraph to be unstable.

\begin{theorem}\label{Thm-3}
A connected digraph $G$ is unstable if and only if there exist distinct permutations $\alpha$ and $\beta$ of $V(G)$ such that, for $u,v\in V(G)$,
\begin{equation}\label{Eqn-1}
u\to_G v\;\Leftrightarrow\;u^\alpha\to_G v^\beta\;\Leftrightarrow\;u^\beta\to_G v^\alpha.
\end{equation}
\end{theorem}

The following corollary provides a useful sufficient condition for a digraph to be unstable, which is easy to apply in practice (for example, in the proof of Theorem~\ref{Thm-6}).

\begin{corollary}\label{Cor-1}
Let $G$ be a connected digraph, and let $\gamma$ be a nontrivial automorphism of $G$.
Suppose that $\{X,Y\}$ is a partition of $V(G)$ such that the following two conditions hold:
\begin{enumerate}[\rm(a)]
\item\label{Cor-1-a} $X^\gamma=X$ and $Y^\gamma=Y$;
\item\label{Cor-1-b} if $\{u,v\}\subseteq X$ or $\{u,v\}\subseteq Y$, then $u\to_G v\;\Rightarrow\;u\to_G v^\gamma$ and $u\to_G v\;\Rightarrow\;u^\gamma\to_G v$.
\end{enumerate}
Then $G$ is unstable.
\end{corollary}

Given a group $R$ and a subset $S$ of $R$, the \emph{Cayley digraph} of $R$ with \emph{connection set} $S$, denoted by $\Cay(R,S)$, is the digraph with vertex set $R$ and arc set $\{(r,sr)\mid r\in R,\,s\in S\}$.
Cayley digraphs of cyclic groups are called \emph{circulant}.
An unstable digraph $G$ is said to be \emph{nontrivially unstable} if $G$ is connected, non-bipartite and twin-free.
In~\cite[Theorems~C.1--C.4]{Wilson2008}, Wilson provided four sufficient conditions for a circulant graph to be unstable; however, it was shown in~\cite{QXZ2019} that the claim in~\cite[Theorem~C.2]{Wilson2008} is not true, and a corrected condition was provided.
In~\cite{HMM2021}, a slight correction of~\cite[Theorem~C.3]{Wilson2008} was made, and~\cite[Theorem~1.4]{HMM2021} gives the best sufficient conditions as far as we know.
The following theorem states four sufficient conditions for a circulant digraph to be unstable, which are similar to those in~\cite[Theorem~1.4]{HMM2021}.
Throughout the paper, $Z_n$ denotes the additive group of integers modulo $n$, and the complete graph $K_n$ has vertex set $Z_n$.

\begin{theorem}\label{Thm-6}
Let $G=\Cay(Z_n,S)$ be a circulant digraph with $n$ even, let $S_0=S\cap2Z_n$, and let $S_1=S\setminus S_0$.
If any of the following conditions is true, then $G$ is unstable:
\begin{enumerate}[\rm(C.1)]
\item there exists a nonzero element $h$ of $2Z_n$ such that $S_0+h=S_0$;
\item $n$ is divisible by $4$, and there exists $h\in1+2Z_n$ such that $S_1+2h=S_1$ and $s+h\in S$ for each $s\in S$ with $s\equiv0$ or $-h\pmod{4}$;
\item there exists a subgroup $H$ of $Z_n$ such that the set $R=\{s\in S\mid H+s\nsubseteq S\}$ is nonempty, and the integer $d=\gcd(R\cup\{n\})$ satisfies the conditions that $n/d$ is even, that $r/d$ is odd for each $r\in R$, and that either $H\nsubseteq dZ_n$ or $H\subseteq2dZ_n$;
\item there exists an integer $h$ coprime to $n$ such that $hS+n/2=S$.
\end{enumerate}
\end{theorem}

Wilson originally conjectured in~\cite{Wilson2008} that the four sufficient conditions therein are also necessary for a circulant graph to be nontrivially unstable.
However, in~\cite{QXZ2019}, it was shown that, even with the corrected condition, the conjecture is false, leaving the problem of classifying all nontrivially unstable circulants open.
We would like to put forward the following similar problem for circulant digraphs.

\begin{problem}
Characterize nontrivially unstable finite circulant digraphs.
\end{problem}

A digraph $G$ is \emph{arc-transitive} if $\Aut(G)$ acts transitively on $A(G)$.
Analogous to the result given in~\cite[Theorem~1.6]{QXZ2019} for circulant graphs, we have the following theorem.

\begin{theorem}\label{Thm-4}
There exist no nontrivially unstable finite arc-transitive circulant digraphs.
\end{theorem}

\begin{remark}
As we will show in Lemma~\ref{Thm-2}, if $H$ is undirected and $(\underline{G},H)$ is stable, then $(G,H)$ is stable.
In particular, when $H=K_2$, the stability of $\underline{G}$ implies the stability of $G$.
However, Theorem~\ref{Thm-4} does not follow directly from Lemma~\ref{Thm-2} and~\cite[Theorem~1.6]{QXZ2019} since $\underline{G}$ is not necessarily twin-free even if $G$ is.
For example, $G=\Cay(Z_{16},\{1,7\})$ is twin-free, whereas $\underline{G}$ is not.
\end{remark}

The following theorem is attributed to Dave Morris, as the proof for the special case of Cayley graphs in his paper~\cite{Morris2021} extends verbatim to the more general setting of digraphs, which is the formulation adopted here. For the convenience of the reader and for the sake of self-containment, we nevertheless provide a complete proof, which is short and, in our presentation, may be slightly more accessible to our intended audience.

\begin{theorem}\label{Thm-5}
There exist no nontrivially unstable Cayley digraphs of abelian groups of odd order.
\end{theorem}

\begin{remark}
It is worth noting that Theorem~\ref{Thm-5} is not a direct consequence of Lemma~\ref{Thm-2} and~\cite[Theorem~1.1]{Morris2021}, as $\underline{G}$ is not necessarily twin-free given $G$ twin-free.
For example, $G=\Cay(Z_9,\{1,2,4\})$ is twin-free, while $\underline{G}$ is not.
\end{remark}

\section{Skeleton}

We call an out-neighbor or in-neighbor a \emph{neighbor}. The neighborhood (set of neighbors) of a vertex $v$ in a digraph $G$ is denoted by $G(v)$.

\begin{definition}\label{Def-1}
For a digraph $G$, let $B(G)$ be the simple graph with vertex set $V(G)$ such that two distinct vertices $u$ and $v$ are adjacent if and only if either $G^+(u)\cap G^+(v)\neq\varnothing$ or $G^-(u)\cap G^-(v)\neq\varnothing$. An edge $\{u,v\}$ of $B(G)$ is called \emph{$G$-dispensable} if there exist $w\in V(G)$ and $\epsilon\in\{+,-\}$ such that one of the following holds:
\begin{enumerate}[{\rm(a)}]
\item\label{Def-1-a} $G^\epsilon(u)\subsetneq G^\epsilon(w)\subsetneq G^\epsilon(v)$,
\item\label{Def-1-b} $G^\epsilon(v)\subsetneq G^\epsilon(w)\subsetneq G^\epsilon(u)$,
\item\label{Def-1-c} $G^\epsilon(u)\cap G^\epsilon(v)\subsetneq G^\epsilon(u)\cap G^\epsilon(w)$ and $G^\epsilon(u)\cap G^\epsilon(v)\subsetneq G^\epsilon(v)\cap G^\epsilon(w)$.
\end{enumerate}
The \emph{Cartesian skeleton} $S(G)$ of a digraph $G$ is the spanning subgraph of $B(G)$ obtained by removing all the $G$-dispensable edges from $B(G)$.
\end{definition}

As an example for Definition~\ref{Def-1} that will be needed in the sequel, $S(K_n)=B(K_n)=K_n$ for each integer $n\geq3$.

\begin{lemma}\label{Lem-6}
Let $G$ be a digraph, and let $\alpha$ and $\beta$ be permutations of $V(G)$ satisfying~\eqref{Eqn-1} for $u,v\in V(G)$. Then the following statements hold:
\begin{enumerate}[{\rm(a)}]
\item\label{Lem-6-a} $G^\epsilon(w^\alpha)=(G^\epsilon(w))^\beta$ and $G^\epsilon(w^\beta)=(G^\epsilon(w))^\alpha$ for each $\epsilon\in\{+,-\}$ and $w\in V(G)$;
\item\label{Lem-6-b} $\alpha,\beta\in\Aut(B(G))$;
\item\label{Lem-6-c} $\alpha,\beta\in\Aut(S(G))$.
\end{enumerate}
\end{lemma}

\begin{proof}
Let $x\in G^\epsilon(w^\alpha)$. Then $x=u^{\beta}$ for some $u\in V(G)$, as $\beta$ is a permutation of $V(G)$. By~\eqref{Eqn-1}, we derive from $u^\beta=x\in G^\epsilon(w^\alpha)$ that $u\in G^\epsilon(w)$ and hence $x=u^\beta\in(G^\epsilon(w))^\beta$. This shows that $G^\epsilon(w^\alpha)\subseteq(G^\epsilon(w))^\beta$.
Conversely, let $y\in(G^\epsilon(w))^\beta$. Then $y=v^\beta$ for some $v\in G^\epsilon(w)$, and so by~\eqref{Eqn-1}, we obtain $y=v^\beta\in G^\epsilon(w^\alpha)$. This shows that $(G^\epsilon(w))^\beta\subseteq G^\epsilon(w^\alpha)$, whence $G^\epsilon(w^\alpha)=(G^\epsilon(w))^\beta$. Similarly, $G^\epsilon(w^\beta)=(G^\epsilon(w))^\alpha$, which proves statement~\eqref{Lem-6-a}.

For $u,v\in V(G)$, by statement~\eqref{Lem-6-a}, we have
\begin{align*}
\{u,v\}\in E(B(G))&\;\Leftrightarrow\;G^\epsilon(u)\cap G^\epsilon(v)\neq\varnothing\,\text{ for some }\epsilon\in\{+,-\}\\
&\;\Leftrightarrow\;(G^\epsilon(u)\cap G^\epsilon(v))^\beta\neq\varnothing\,\text{ for some }\epsilon\in\{+,-\}\\
&\;\Leftrightarrow\;(G^\epsilon(u))^\beta\cap(G^\epsilon(v))^\beta\neq\varnothing\,\text{ for some }\epsilon\in\{+,-\}\\
&\;\Leftrightarrow\;G^\epsilon(u^\alpha)\cap G^\epsilon(v^\alpha)\neq\varnothing\,\text{ for some }\epsilon\in\{+,-\}\\
&\;\Leftrightarrow\;\{u^\alpha,v^\alpha\}\in E(B(G)).
\end{align*}
Thus, $\alpha\in\Aut(B(G))$. Similarly, $\beta\in\Aut(B(G))$, and so statement~\eqref{Lem-6-b} holds.

Now that $\alpha\in\Aut(B(G))$, to prove $\alpha\in\Aut(S(G))$, it suffices to show that an edge $\{u,v\}$ of $B(G)$ is $G$-dispensable if and only if the edge $\{u^\alpha,v^\alpha\}$ of $B(G)$ is $G$-dispensable. This can be seen by applying $\beta$ to~\eqref{Def-1-a}--\eqref{Def-1-c} of Definition~\ref{Def-1}, which gives
\begin{gather*}
G^\epsilon(u)\subsetneq G^\epsilon(w)\subsetneq G^\epsilon(v),\\
G^\epsilon(v)\subsetneq G^\epsilon(w)\subsetneq G^\epsilon(u),\\
G^\epsilon(u)\cap G^\epsilon(v)\subsetneq G^\epsilon(u)\cap G^\epsilon(w)\ \text{ and }\ G^\epsilon(u)\cap G^\epsilon(v)\subsetneq G^\epsilon(v)\cap G^\epsilon(w),
\end{gather*}
respectively, by statement~\eqref{Lem-6-a} of the lemma. Similarly, $\beta\in\Aut(S(G))$, proving statement~\eqref{Lem-6-c}.
\end{proof}

\begin{definition}
A digraph $G$ is \emph{strongly twin-free} if there are no distinct vertices $u$ and $v$ of $G$ such that $G^+(u)=G^+(v)$ or $G^-(u)=G^-(v)$.
\end{definition}

The \emph{Cartesian product} of digraphs $G$ and $H$ is a digraph, denoted as $G\Bbox H$, such that the vertex set is $V(G)\times V(H)$, and $(u,i)\to_{G\Bbox H}(v,j)$ if and only if either $u=v$ and $i\to_H j$ or $u\to_G v$ and $i=j$.

\begin{lemma}\label{Lem-3}
Let $G$ and $H$ be strongly twin-free digraphs such that neither $G$ nor $H$ has a vertex with an empty out-neighborhood or in-neighborhood. Then $S(G\times H)=S(G)\Bbox S(H)$.
\end{lemma}

\begin{proof}
For convenience, call conditions~\eqref{Def-1-a}--\eqref{Def-1-c} of Definition~\ref{Def-1} (a')--(c') if we replace $G$ by $G\times H$ and replace $u$, $v$ and $w$ by $x$, $y$ and $z$, respectively.
Since both $S(G\times H)$ and $S(G)\Bbox S(H)$ have vertex set $V(G)\times V(H)$, the proof will be complete after the two steps below.

\textsf{Step~1:} show that each edge $\{x,y\}$ of $S(G\times H)$ is an edge of $S(G)\Bbox S(H)$.

Write $x=(u,i)\in V(G)\times V(H)$ and $y=(v,j)\in V(G)\times V(H)$.
Since $\{x,y\}$ is an edge of $S(G\times H)$ and hence an edge of $B(G\times H)$, there exists $\epsilon\in\{+,-\}$ such that $(G\times H)^\epsilon(x)\cap(G\times H)^\epsilon(y)\neq\varnothing$, which means
\begin{equation}\label{Eqn-2}
G^\epsilon(u)\cap G^\epsilon(v)\neq\varnothing\ \text{ and }\ H^\epsilon(i)\cap H^\epsilon(j)\neq\varnothing.
\end{equation}
Since $\{x,y\}$ is not $(G\times H)$-dispensable, there exists no $z\in V(G\times H)$ satisfying all the conditions~(a')--(c'). Then noting~\eqref{Eqn-2} and
\begin{align*}
(G\times H)^\epsilon(x)\cap(G\times H)^\epsilon(y)&=(G^\epsilon(u)\cap G^\epsilon(v))\times(H^\epsilon(i)\cap H^\epsilon(j))\\
&\subseteq G^\epsilon(u)\times(H^\epsilon(i)\cap H^\epsilon(j))=(G\times H)^\epsilon(x)\cap(G^\epsilon\times H^\epsilon)(z),\\
(G\times H)^\epsilon(x)\cap(G\times H)^\epsilon(y)&=(G^\epsilon(u)\cap G^\epsilon(v))\times(H^\epsilon(i)\cap H^\epsilon(j))\\
&\subseteq(G^\epsilon(u)\cap G^\epsilon(v))\times H^\epsilon(j)=(G\times H)^\epsilon(y)\cap(G^\epsilon\times H^\epsilon)(z),
\end{align*}
we conclude that
\begin{equation}\label{Eqn-3}
G^\epsilon(u)\cap G^\epsilon(v)=G^\epsilon(u)\ \text{ or }\ H^\epsilon(i)\cap H^\epsilon(j)=H^\epsilon(j)
\end{equation}
(otherwise, $z\coloneqq(u,j)$ would satisfy (c')). In the same vein (swapping $u$ with $v$ and $i$ with $j$),
\begin{equation}\label{Eqn-4}
G^\epsilon(u)\cap G^\epsilon(v)=G^\epsilon(v)\ \text{ or }\ H^\epsilon(i)\cap H^\epsilon(j)=H^\epsilon(i).
\end{equation}
Now the argument proceeds according to the choice of cases for each of~\eqref{Eqn-3} and~\eqref{Eqn-4}.

First assume that $G^\epsilon(u)\cap G^\epsilon(v)=G^\epsilon(u)$ and $G^\epsilon(u)\cap G^\epsilon(v)=G^\epsilon(v)$.
Since $G$ is strongly twin-free, this implies $u=v$. We prove $\{i,j\}\in E(S(H))$, which will show that $\{x,y\}$ is an edge of $S(G)\Bbox S(H)$. Suppose for a contradiction that $\{i,j\}\notin E(S(H))$. Then as~\eqref{Eqn-2} indicates $\{i,j\}\in E(B(H))$, it follows that $\{i,j\}$ is $H$-dispensable. Hence there exists $\delta\in\{+,-\}$ and $k\in V(H)$ such that one of the following holds:
\begin{gather*}
H^\delta(i)\subsetneq H^\delta(k)\subsetneq H^\delta(j),\\
H^\delta(j)\subsetneq H^\delta(k)\subsetneq H^\delta(i),\\
H^\delta(i)\cap H^\delta(j)\subsetneq H^\delta(u)\cap H^\delta(k)\ \text{ and }\ H^\delta(i)\cap H^\delta(j)\subsetneq H^\delta(j)\cap H^\delta(k).
\end{gather*}
Multiplying the set $G^\delta(u)$ to the left as a factor of the Cartesian product (of sets), we obtain (a')--(c') respectively with $z\coloneqq(u,k)$ and $\epsilon$ replaced by $\delta$. This implies that $\{x,y\}$ is $(G\times H)$-dispensable, a contradiction. Thus, $\{i,j\}\in E(S(H))$, as desired.

Next assume that $G^\epsilon(u)\cap G^\epsilon(v)=G^\epsilon(u)$ and $H^\epsilon(i)\cap H^\epsilon(j)=H^\epsilon(i)$. Since $G$ and $H$ are strongly twin-free, it follows that $G^\epsilon(u)\subsetneq G^\epsilon(v)$ and $H^\epsilon(i)\subsetneq H^\epsilon(j)$. Therefore,
\[
(G\times H)^\epsilon(x)=G^\epsilon(u)\times H^\epsilon(i)\subsetneq G^\epsilon(v)\times H^\epsilon(i)\subsetneq G^\epsilon(v)\times H^\epsilon(j)=(G\times H)^\epsilon(y).
\]
In other words, $z\coloneqq(v,i)$ satisfies (b'), a contradiction.

The remaining two cases are dealt with similarly. This completes Step~1.

\textsf{Step~2:} show that each edge $\{x,y\}$ of $S(G)\Bbox S(H)$ is an edge of $S(G\times H)$.

Write $x=(u,i)\in V(G)\times V(H)$ and $y=(v,j)\in V(G)\times V(H)$.
Since $\{x,y\}$ is an edge of $S(G)\Bbox S(H)$, either $u=v$ or $i=j$.
Without loss of generality, assume $i=j$. Then $\{u,v\}$ is an edge of $B(G)$ that is not $G$-dispensable.
In particular, there exists $\delta\in\{+,-\}$ such that $G^\delta(u)\cap G^\delta(v)\neq\varnothing$, and so
\[
(G\times H)^\delta(x)\cap(G\times H)^\delta(y)=(G^\delta(u)\cap G^\delta(v))\times H^\delta(i)\neq\varnothing
\]
as $H^\delta(i)\neq\varnothing$. Hence $\{x,y\}$ is an edge of $B(G\times H)$.
Suppose for a contradiction that $\{x,y\}$ is $(G\times H)$-dispensable. Then there exist $z=(w,k)\in V(G)\times V(H)$ and $\epsilon\in\{+,-\}$ satisfying one of~(a')--(c').

First assume that $z$ and $\epsilon$ satisfy~(a'); that is,
\[
G^\epsilon(u)\times H^\epsilon(i)\subsetneq G^\epsilon(w)\times H^\epsilon(k)\subsetneq G^\epsilon(v)\times H^\epsilon(i).
\]
Since $G^\epsilon(u)\neq\varnothing$, we deduce $H^\epsilon(i)\subseteq H^\epsilon(k)\subseteq H^\epsilon(i)$, which in turn gives $G^\epsilon(u)\subsetneq G^\epsilon(w)\subsetneq G^\epsilon(v)$. This is not possible as $\{u,v\}$ is not $G$-dispensable.

For the same reason, $z$ and $\epsilon$ cannot satisfy~(b'), and so they satisfy~(c'), namely,
\begin{align*}
&(G\times H)^\epsilon((u,i))\cap(G\times H)^\epsilon((v,i))\subsetneq(G\times H)^\epsilon((u,i))\cap(G\times H)^\epsilon((w,k))\\
&(G\times H)^\epsilon((u,i))\cap(G\times H)^\epsilon((v,i))\subsetneq(G\times H)^\epsilon((v,i))\cap(G\times H)^\epsilon((w,k)).
\end{align*}
It follows that
\begin{align*}
&(G^\epsilon(u)\cap G^\epsilon(v))\times H^\epsilon(i)\subsetneq(G^\epsilon(u)\cap G^\epsilon(w))\times(H^\epsilon(i)\cap H^\epsilon(k)),\\
&(G^\epsilon(u)\cap G^\epsilon(v))\times H^\epsilon(i)\subsetneq(G^\epsilon(v)\cap G^\epsilon(w))\times(H^\epsilon(i)\cap H^\epsilon(k)).
\end{align*}
Since $H^\epsilon(i)\supseteq H^\epsilon(i)\cap H^\epsilon(k)$, we derive
\[
G^\epsilon(u)\cap G^\epsilon(v)\subsetneq G^\epsilon(u)\cap G^\epsilon(w)\ \text{ and } \ G^\epsilon(u)\cap G^\epsilon(v)\subsetneq G^\epsilon(v)\cap G^\epsilon(w),
\]
which is again not possible as $\{u,v\}$ is not $G$-dispensable.

Thereby we conclude that $\{x,y\}$ is not $(G\times H)$-dispensable. Consequently, $\{x,y\}$ is an edge of $S(G\times H)$.
\end{proof}

\section{Partition}

\begin{definition}\label{Def-2}
For digraphs $G$ and $H$, the \emph{$G$-partition} is the partition $\{\{u\}\times V(H)\mid u\in V(G)\}$ of $V(G\times H)$, and the \emph{$H$-partition} is the partition $\{V(G)\times\{i\}\mid i\in V(H)\}$ of $V(G\times H)$. Define the following subgroups of $\Aut(G\times H)$:
\begin{enumerate}[{\rm(a)}]
\item\label{Def-2-a} let $\Aut_G(G\times H)$ consist of $\alpha\in\Aut(G\times H)$ that preserves the $G$-partition;
\item\label{Def-2-b} let $\Aut_H(G\times H)$ consist of $\alpha\in\Aut(G\times H)$ that preserves the $H$-partition;
\item\label{Def-2-c} let $\Aut_G(H)$ consist of $\alpha\in\Aut(G\times H)$ that stabilizes each part of the $G$-partition;
\item\label{Def-2-d} let $\Aut_H(G)$ consist of $\alpha\in\Aut(G\times H)$ that stabilizes each part of the $H$-partition.
\end{enumerate}
Thus, $\Aut_G(H)$ is the kernel of $\Aut_G(G\times H)$ acting on the $G$-partition, and $\Aut_H(G)$ is the kernel of $\Aut_H(G\times H)$ acting on the $H$-partition. Via the bijection $\{u\}\times V(H)\mapsto u$ between the $G$-partition and $V(G)$, the group $\Aut_G(G\times H)$ acts on $V(G)$ with kernel $\Aut_G(H)$. Similarly, via the natural bijection between the $H$-partition and $V(H)$, the group $\Aut_H(G\times H)$ acts on $V(H)$ with kernel $\Aut_H(G)$.
\end{definition}

\begin{lemma}\label{Lem-2}
Let $G$ and $H$ be digraphs with nonempty arc sets. Viewing $\Aut(G)$ as the subgroup $\Aut(G)\times1$ of $\Aut(G\times H)$ and viewing $\Aut(H)$ as the subgroup $1\times\Aut(H)$ of $\Aut(G\times H)$, the following statements hold:
\begin{enumerate}[\rm(a)]
\item\label{Lem-2-a} $\Aut(G)\leq\Aut_H(G)$;
\item\label{Lem-2-b} $\Aut(H)\leq\Aut_G(H)$;
\item\label{Lem-2-c} $\Aut_G(G\times H)=\Aut_G(H)\rtimes\Aut(G)$;
\item\label{Lem-2-d} $\Aut_H(G\times H)=\Aut_H(G)\rtimes\Aut(H)$.
\end{enumerate}
In particular, the induced permutation groups of $\Aut_G(G\times H)$ and $\Aut_H(G\times H)$ on $V(G)$ and $V(H)$, respectively, are $\Aut(G)$ and $\Aut(H)$.
\end{lemma}

\begin{proof}
By the definition of $\Aut_H(G)$ and $\Aut_G(H)$, we immediately have statements~\eqref{Lem-2-a} and~\eqref{Lem-2-b}, as well as $\Aut_G(H)\cap\Aut(G)=1$. Since $\Aut_G(H)$ is the kernel of $\Aut_G(G\times H)$ acting on the $V(G)$, the subgroup $\Aut_G(H)$ is normal in $\Aut_G(G\times H)$. Thus, to prove statement~\eqref{Lem-2-c}, it suffices to show that the induced permutation group of $\Aut_G(G\times H)$ on $V(G)$ is $\Aut(G)$. For this purpose, fix an arc $(i,j)$ of $H$ and consider an arbitrary $\alpha\in\Aut_G(G\times H)$, with the induced permutation $\sigma$ on $V(G)$. Then for $u,v\in V(G)$,
\begin{align*}
u\to v\;&\Rightarrow\;(u,i)\to(v,j)\;\Rightarrow\;(u,i)^\alpha\to(v,j)^\alpha\;\Rightarrow\;u^\sigma\to v^\sigma,\\
u\to v\;&\Rightarrow\;(u,i)\to(v,j)\;\Rightarrow\;(u,i)^{\alpha^{-1}}\to(v,j)^{\alpha^{-1}}\;\Rightarrow\;u^{\sigma^{-1}}\to v^{\sigma^{-1}}.
\end{align*}
This shows that $\sigma\in\Aut(G)$, as desired, and similarly we obtain statement~\eqref{Lem-2-d}.
\end{proof}

\begin{lemma}\label{Lem-10}
If a pair $(G,H)$ of digraphs is stable, then $\Aut_H(G)=\Aut(G)$ and $\Aut_G(H)=\Aut(H)$.
\end{lemma}

\begin{proof}
By Lemma~\ref{Lem-2}, $\Aut(G)\times\Aut(H)\leq\Aut_H(G)\rtimes\Aut(H)=\Aut_H(G\times H)\leq\Aut(G\times H)$. Hence, if $(G,H)$ is stable, then the equality $\Aut(G\times H)=\Aut(G)\times\Aut(H)$ forces $\Aut_H(G)=\Aut(G)$. Similarly, if $(G,H)$ is stable, then $\Aut_G(H)=\Aut(H)$.
\end{proof}

\begin{lemma}\label{Lem-1}
Let $G$ and $H$ be digraphs with nonempty arc sets. Then $(G,H)$ is stable if and only if $\Aut(G\times H)=\Aut_G(G\times H)=\Aut_H(G\times H)$.
\end{lemma}

\begin{proof}
If $(G,H)$ is stable, then $\Aut(G\times H)=\Aut(G)\times\Aut(H)$ preserves both the $G$-partition and $H$-partition, and so $\Aut(G\times H)=\Aut_G(G\times H)=\Aut_H(G\times H)$. Conversely, suppose $\Aut(G\times H)=\Aut_G(G\times H)=\Aut_H(G\times H)$, which means that $\Aut(G\times H)$ preserves both the $G$-partition and $H$-partition. Then each element of $\Aut(G\times H)$ has the form $(\alpha,\beta)$, where $\alpha$ and $\beta$ are permutations on $V(G)$ and $V(H)$ respectively, we deduce from Lemma~\ref{Lem-2} that $\alpha\in\Aut(G)$ and $\beta\in\Aut(H)$. It follows that $\Aut(G\times H)=\Aut(G)\times\Aut(H)$, as required.
\end{proof}

\begin{remark}
Elements of $\Aut(G\times H)\setminus\Aut_G(G\times H)$ are called \emph{$G$-mixers}, while elements of $\Aut(G\times H)\setminus\Aut_H(G\times H)$ are called \emph{$H$-mixers}  (as coined in~\cite[Definition~2.2]{GLX2025}). Either an $H$-mixer or a $G$-mixer is simply called a \emph{mixer} of the digraph pair $(G,H)$. Lemma~\ref{Lem-1} says that a digraph pair fails to be stable if and only if it has mixers.
\end{remark}

\begin{proof}[\bf\textup{Proof of Theorem~$\ref{Thm-1}$}]
To prove statement~\eqref{Thm-1-a}, suppose for a contradiction that $G$ and $H$ are not coprime. Then there exist digraphs $G_0$, $H_0$ and $K$ with $|V(K)| > 1$ such that $G = G_0 \times K$ and $H = H_0 \times K$. Accordingly, $G \times H = G_0 \times K \times H_0 \times K$. Let $\alpha$ be the permutation of $V(G \times H) = V(G_0) \times V(K) \times V(H_0) \times V(K)$ defined by
\[
(u, i, v, j)^\alpha = (u, j, v, i)
\]
for $(u, i, v, j) \in V(G_0) \times V(K) \times V(H_0) \times V(K)$. Clearly, $\alpha\in\Aut(G\times H)$. Since $|V(K)| > 1$, we have $\alpha\notin \Aut_G(G \times H)$. This contradicts Lemma~\ref{Lem-1} as $(G,H)$ is stable.

To prove statement~\eqref{Thm-1-b}, suppose for a contradiction that one of $G$ or $H$, say, $G$, is not twin-free. Then there exist distinct vertices $u$ and $v$ of $G$ such that $G^+(u)=G^+(v)$ and $G^-(u)=G^-(v)$. Let $\tau$ be the permutation of $V(G)$ which swaps $u$ and $v$ and fixes each vertex in $V(G)\setminus\{u, v\}$, and fix any $i\in V(H)$. Let $\alpha$ be the permutation of $V(G)\times V(H)$ defined by
\[
(w,j)^\alpha=
\begin{cases}
(w^\tau,j) & \mbox{if } j=i \\
(w,j) & \mbox{if } j\neq i.
\end{cases}
\]
It is straightforward to verify that $\alpha\in\Aut_H(G)\setminus\Aut(G)$. Thus, by Lemma~\ref{Lem-10}, $(G,H)$ is unstable, a contradiction.

Next, we prove statement~\eqref{Thm-1-c} and suppose for a contradiction that both $\Aut(G)$ and $\Aut(H)$ are nontrivial while one of $G$ or $H$ is disconnected. Without loss of generality, assume that $G$ is disconnected. Take any connected component $G_0$ of $G$ and take any non-identity $\sigma\in\Aut(H)$. Let $\alpha$ be the permutation of $V(G)\times V(H)$ defined by
\[
(v,i)^\alpha=
\begin{cases}
(v,i^\sigma) & \mbox{if } v\in V(G_0)\\
(v,i) & \mbox{if } v\in V(G)\setminus V(G_0).
\end{cases}
\]
It is direct to verify that $\alpha\in\Aut_G(H)\setminus\Aut(H)$. Then we again derive a contradiction from Lemma~\ref{Lem-10} that $(G,H)$ is unstable.

For the rest of the proof, we embark on statement~\eqref{Thm-1-d} and suppose that both $G$ and $H$ are bipartite. By~\eqref{Thm-1-c}, both $G$ and $H$ are connected. Let $\{B_1, B_2\}$ be the bipartition of $\underline{G}$ and $\{C_1, C_2\}$ the bipartition of $\underline{H}$, and let $D= (B_1 \times C_1) \cup (B_2 \times C_2)$ and $E= (B_1 \times C_2) \cup (B_2 \times C_1)$. Then $\{D,E\}$ is a partition of $V(G)\times V(H)$, and there is no arc of $G\times H$ between $D$ and $E$.

Take an arbitrary automorphism $\gamma$ of the induced subdigraph $(G\times H)[D]$. Let $\alpha$ be the permutation of $V(G)\times V(H)$ defined by
\begin{equation}\label{Eqn-12}
(v, i)^\alpha=
\begin{cases}
(v, i)^\gamma & \text{if } (v, i) \in D \\
(v, i) & \text{if } (v, i) \in E.
\end{cases}
\end{equation}
Then $\alpha\in\Aut(G \times H)$, and we conclude from Lemma~\ref{Lem-1} that $\alpha\in\Aut_G(G\times H)\cap\Aut_H(G\times H)$. This together with the second line of~\eqref{Eqn-12} implies that $\alpha$ is the identity. Consequently, $\gamma$ is the identity.

Now we see that $\Aut((G\times H)[D])=1$. Similarly, $\Aut((G\times H)[E])=1$. Since both $\Aut(G)$ and $\Aut(H)$ are nontrivial, there exist non-identity $\lambda$ and $\mu$ in $\Aut(G)$ and $\Aut(H)$ respectively.
If $\lambda$ stabilizes both $B_1$ and $B_2$, then
\[
D^{(\lambda,1)}=(B_1^\lambda\times C_1^1)\cup(B_2^\lambda\times C_2^1)=(B_1\times C_1)\cup(B_2\times C_2)=D,
\]
which together with $\Aut((G\times H)[D])=1$ and $\Aut((G\times H)[E])=1$ implies that $(\lambda,1)$ is identity, a contradiction.
Hence $\lambda$ swaps $B_1$ and $B_2$ as $G$ is bipartite with bipartition $\{B_1,B_2\}$.
Similarly, $\mu$ swaps $C_1$ and $C_2$.
It follows that
\[
D^{(\lambda,\mu)}=(B_1^\lambda\times C_1^\mu)\cup(B_2^\lambda\times C_2^\mu)=(B_2\times C_2)\cup(B_1\times C_1)=D.
\]
This combined with $\Aut((G\times H)[D])=1$ and $\Aut((G\times H)[E])=1$ implies that $(\lambda,\mu)$ is identity, a contradiction.
This completes the proof.
\end{proof}

\begin{lemma}\label{Lem-5}
Let $G$ be a digraph, let $H$ be a graph, and let $K$ and $L$ be connected graphs of order at least $2$. The following statements hold:
\begin{enumerate}[\rm(a)]
\item\label{Lem-5-a} $\underline{G\times H}=\underline{G}\times H$ and $\underline{H\times G}=H\times\underline{G}$;
\item\label{Lem-5-b} If at least one of $K$ or $L$ is non-bipartite, then $K\times L$ is connected.
\end{enumerate}
\end{lemma}

\begin{proof}
Let $\{(u,i),(v,j)\}$ be an edge of $\underline{G\times H}$. Then either $(u,i)\to_{G\times H}(v,j)$ or $(v,j)\to_{G\times H}(u,i)$. In either case, $\{u,v\}\in E(\underline{G})$ and $\{i,j\}\in E(H)$, from which we deduce that $(u,i)$ and $(v,j)$ are adjacent in $\underline{G}\times H$.
Conversely, let $\{(u,i),(v,j)\}$ be an edge of $\underline{G}\times H$. Then $\{u,v\}\in E(\underline{G})$ and $\{i,j\}\in E(H)$. The former implies $u\to_G v$ or $v\to_G u$, while the latter implies $i\to_H j$ and $j\to_H i$. Thus, either $(u,i)\to_{G\times H}(v,j)$ or $(v,j)\to_{G\times H}(u,i)$. This means that $(u,i)$ and $(v,j)$ are adjacent in $\underline{G\times H}$, and hence $\underline{G\times H}=\underline{G}\times H$. Similarly, $\underline{H\times G}=H\times\underline{G}$, proving statement~\eqref{Lem-5-a}.

Statement~\eqref{Lem-5-b} follows from~\cite{Weichsel1962} (see also~\cite[Theorem~5.9]{HIK2011}).
\end{proof}

\begin{lemma}\label{Thm-2}
Let $G$ and $H$ be digraphs with nonempty arc sets. If $H$ is undirected and $(\underline{G},H)$ is stable, then $(G,H)$ is stable.
\end{lemma}

\begin{proof}
Suppose that $H$ is undirected and $(\underline{G},H)$ is stable. Take an arbitrary $\alpha\in\Aut(G\times H)$. Then $\alpha$ is automorphism of $\underline{G\times H}$ and hence an automorphism of $\underline{G}\times H$ by Lemma~\ref{Lem-5}. Since $(\underline{G},H)$ is stable, it follows that $\alpha\in\Aut(\underline{G})\times\Aut(H)$. In particular, $\alpha$ preserves both the $G$-partition and $H$-partition. Thus, by Lemma~\ref{Lem-1}, $(G,H)$ is stable.
\end{proof}

\begin{remark}
The converse of Lemma~\ref{Thm-2} is not true. For example, take $G$ to be the directed path of length $2$ and $H$ to be $K_2$.
\end{remark}

\begin{proof}[\bf\textup{Proof of Theorem~$\ref{Thm-3}$}]
Denote $H=K_2$ with $V(H)=\{0,1\}$. By Lemma~\ref{Lem-5}, $\underline{G\times H}=\underline{G}\times H$ is connected and bipartite as $\underline{G}$ is connected. Thus, since each automorphism of $G\times H$ is an automorphism of $\underline{G\times H}$, it preserves the $H$-partition $\{V(G)\times\{0\},V(G)\times\{1\}\}$. Hence $\Aut(G\times H)=\Aut_H(G\times H)$, which together with Lemma~\ref{Lem-2}\,\eqref{Lem-2-d} leads to $\Aut(G\times H)=\Aut_H(G)\rtimes\Aut(H)$. Since, by definition, $G$ is stable if and only if $\Aut(G\times H)=\Aut(G)\times\Aut(H)$, it follows that
$G$ is unstable if and only if
\[
\Aut_H(G)>\Aut(G).
\]

Suppose that $G$ is unstable. Then there exists $\gamma\in\Aut_H(G)\setminus\Aut(G)$. Since $\gamma$ lies in the kernel $\Aut_H(G)$ of $\Aut_H(G\times H)$ on the $H$-partition, it stabilizes each of the two parts $V(G)\times\{0\}$ and $V(G)\times\{1\}$. Let $\alpha$ and $\beta$ be the induced permutations of $\gamma$ on these two parts respectively. Then $\alpha$ and $\beta$ can be naturally viewed as permutations of $V(G)$, in which sense, $\alpha\neq\beta$ as $\gamma\notin\Aut(G)$. Moreover, for $u,v\in V(G)$, we deduce from
\begin{align*}
u\to_G v&\;\Leftrightarrow\;(u,0)\to_{G\times H}(v,1)\;\Leftrightarrow\;(u,0)^\gamma\to_{G\times H}(v,1)^\gamma\;\Leftrightarrow\;(u^\alpha,0)\to_{G\times H}(v^\beta,1)\;\Leftrightarrow\;u^\alpha\to_Gv^\beta,\\
u\to_G v&\;\Leftrightarrow\;(u,1)\to_{G\times H}(v,0)\;\Leftrightarrow\;(u,1)^\gamma\to_{G\times H}(v,0)^\gamma\;\Leftrightarrow\;(u^\beta,1)\to_{G\times H}(v^\alpha,0)\;\Leftrightarrow\;u^\beta\to_Gv^\alpha
\end{align*}
that~\eqref{Eqn-1} holds.

Conversely, suppose that there exist distinct permutations $\alpha$ and $\beta$ of $V(G)$ such that~\eqref{Eqn-1} holds for $u,v\in V(G)$. Let $\gamma$ be the permutation of $V(G)\times\{0,1\}$ defined by $(w,0)^\gamma=(w^\alpha,0)$ and $(w,1)^\gamma=(w^\beta,1)$. For $u,v\in V(G)$, we deduce from~\eqref{Eqn-1} that
\begin{align*}
(u,0)\to_{G\times H}(v,1)&\;\Leftrightarrow\;u\to_G v\;\Leftrightarrow\;u^\alpha\to_Gv^\beta\;\Leftrightarrow\;(u^\alpha,0)\to_{G\times H}(v^\beta,1)\;\Leftrightarrow\;(u,0)^\gamma\to_{G\times H}(v,1)^\gamma,\\
(u,1)\to_{G\times H}(v,0)&\;\Leftrightarrow\;u\to_G v\;\Leftrightarrow\;u^\beta\to_Gv^\alpha\;\Leftrightarrow\;(u^\beta,1)\to_{G\times H}(v^\alpha,0)\;\Leftrightarrow\;(u,1)^\gamma\to_{G\times H}(v,0)^\gamma.
\end{align*}
Hence $\gamma\in\Aut(G\times H)$, and so $\gamma\in\Aut_H(G)$. Moreover, $\gamma\notin\Aut(G)$ as $\alpha\neq\beta$. This shows that $\Aut_H(G)>\Aut(G)$, completing the proof.
\end{proof}

\begin{proof}[\bf\textup{Proof of Corollary~$\ref{Cor-1}$}]
We first prove that, for $\{u,v\}\subseteq X$ or $Y$,
\begin{equation}\label{Eqn-16}
u\to_G v\;\Leftrightarrow\;u\to_G v^\gamma\;\Leftrightarrow\;u^\gamma\to_G v.
\end{equation}
By Corollary~\ref{Cor-1}\,\eqref{Cor-1-b}, we already know that $u\to_G v\Rightarrow u\to_G v^\gamma$ and $u\to_G v\Rightarrow u^\gamma\to_G v$.
Assume that $u\to_G v^\gamma$.
Then $u^{\gamma^{-1}}\to_G v$ as $\gamma\in\Aut(G)$.
Since $X^{\gamma^{-1}}=X$, we have $\{u^{\gamma^{-1}},v\}\subseteq X$, and so $u=(u^{\gamma^{-1}})^\gamma\to_G v$.
Similarly, $u^\gamma\to_G v\Rightarrow u\to_G v$, which completes the proof of~\eqref{Eqn-16}.

Since $\gamma$ is an automorphism of $G$ stabilizing both $X$ and $Y$, the mappings $\alpha$ and $\beta$ defined as follows are permutations of $V(G)$:
\[
u^\alpha=
\begin{cases}
u^\gamma & \mbox{if } u\in X\\
u & \mbox{if } u\in Y
\end{cases}
\quad \mbox{and} \quad
u^\beta=
\begin{cases}
u & \mbox{if } u\in X\\
u^\gamma & \mbox{if } u\in Y.
\end{cases}
\]
Since $\gamma$ is nontrivial, there exists $w\in V(G)$ such that $w^\gamma\neq w$, and so $\alpha$ and $\beta$ are distinct.
By Theorem~\ref{Thm-3}, we only need to prove that for $u,v\in V(G)$,
\begin{equation}\label{Eqn-17}
u\to_G v\;\Leftrightarrow\;u^\alpha\to_G v^\beta\;\Leftrightarrow\;u^\beta\to_G v^\alpha.
\end{equation}
If $\{u,v\}\subseteq X$, then $(u^\alpha,v^\beta)=(u^\gamma,v)$ and $(u^\beta,v^\alpha)=(u,v^\gamma)$, which together with~\eqref{Eqn-16} lead to~\eqref{Eqn-17}.
Similarly,~\eqref{Eqn-17} holds if $\{u,v\}\subseteq Y$.
Now assume that $\{u,v\}\nsubseteq X$ and $\{u,v\}\nsubseteq Y$.
Then $\{(u^\alpha,v^\beta),(u^\beta,v^\alpha)\}=\{(u^\gamma,v^\gamma),(u,v)\}$, and so~\eqref{Eqn-17} holds as $\gamma$ is an automorphism.
\end{proof}

\begin{proof}[\bf\textup{Proof of Theorem~$\ref{Thm-6}$}]
\textsf{Case~1:} (C.1) holds.
Let $\gamma\colon Z_n\to Z_n,x\mapsto x+h$, and let
\[
X=2Z_n=\{x\in Z_n\mid x\equiv0\!\!\pmod{2}\}\ \text{ and }\ Y=Z_n\setminus X=\{x\in Z_n\mid x\equiv1\!\!\pmod{2}\}.
\]
It is clear that $\gamma$ is an automorphism of $G$ and stabilizes $X$ and $Y$ respectively.
For each $s\in S_0$, we derive from $S_0+h=S_0$ that $(u,(u+s)^\gamma)=(u,u+s+h)\in A(G)$ and $(u^\gamma,u+s)=(u+h,u+s)\in A(G)$.
Then for each $(u,v)\in A(G[X]\cup G[Y])$, it follows that $(u,v^\gamma)\in A(G)$ and $(u^\gamma,v)\in A(G)$.
Thus, by Corollary~\ref{Cor-1}, $G$ is unstable.

\textsf{Case~2:} (C.2) holds. 
Let $\gamma$ be the permutation of $V(G)$ defined by
\[
x^\gamma=
\begin{cases}
x+h & \mbox{if } x\mbox{ is even}\\
x-h & \mbox{if } x\mbox{ is odd},
\end{cases}
\]
and let $X=\{x\in Z_n\mid x\equiv0\text{ or }h\pmod{4}\}$ and $Y=\{x\in Z_n\mid x\equiv2\text{ or}-h\pmod{4}\}$.
Then $X^\gamma=X$ and $Y^\gamma=Y$.
Let $(u,u+s)\in A(G)$, where $s\in S$.
If $s\in S_0$, then $(u+s)^\gamma-u^\gamma=(u+s\pm h)-(u\pm h)=s\in S$; if $s\in S_1$, then $(u+s)^\gamma-u^\gamma=(u+s\pm h)-(u\mp h)=s\pm2h\in S$.
In either case, $(u,v)^\gamma\in A(G)$.
This shows that $\gamma\in\Aut(G)$.
Let $(u,u+s)\in A(G[X]\cup G[Y])$, where $s\in S$.
Note that this forces $s\equiv0$ or $\pm h\pmod{4}$.
We now prove that $(u,(u+s)^\gamma)\in A(G)$ and $(u^\gamma,u+s)\in A(G)$ for each $(u,u+s)\in A(G[X]\cup G[Y])$, which together with Corollary~\ref{Cor-1} would imply that $G$ is unstable.

Assume first that $s\equiv0\pmod{4}$.
Then $s+h\in S_1$, and $s-h=s+h-2h\in S_1-2h=S_1$.
Hence $(u+s)^\gamma-u\in\{s+h,s-h\}\subseteq S$ and $(u+s)-u^\gamma\in\{s-h,s+h\}\subseteq S$.
This proves that $(u,(u+s)^\gamma)\in A(G)$ and $(u^\gamma,u+s)\in A(G)$, as desired.

Assume next that $s\equiv h\pmod{4}$.
Then $s-h\in S_0$.
If $u$ is odd, then the assumption $(u,u+s)\in A(G[X]\cup G[Y])$ together with the definition of $X$ and $Y$ yields $s=(u+s)-u\equiv-h\pmod{4}$, a contradiction.
Consequently, $u$ is even.
It follows that $(u+s)^\gamma-u=(u+s-h)-u=s-h\in S$ and $(u+s)-u^\gamma=(u+s)-(u+h)=s-h\in S$.
This implies $(u,(u+s)^\gamma)\in A(G)$ and $(u^\gamma,u+s)\in A(G)$.

Assume finally that $s\equiv-h\pmod{4}$.
Then $s+h\in S_0$.
If $u$ is even, then it follows from the condition $(u,u+s)\in A(G[X]\cup G[Y])$ and the definition of $X$ and $Y$ that $s=(u+s)-u\equiv h\pmod{4}$, a contradiction.
Thus, $u$ is odd.
Hence $(u+s)^\gamma-u=(u+s+h)-u=s+h\in S$, and $(u+s)-u^\gamma=(u+s)-(u-h)=s+h\in S$.
This implies $(u,(u+s)^\gamma)\in A(G)$ and $(u^\gamma,u+s)\in A(G)$.

\textsf{Case~3:} (C.3) holds.
Let $h$ be a generator of $H$, and let $K=dZ_n\leq Z_n$.
Then $|K|=n/d$ is even.
Let $K_0=2K$ and $K_1=K\setminus K_0$.
Then $R\subseteq K_1$ as $r/d$ is odd for each $r\in R$.
Moreover, either $H\nsubseteq K$ or $H\subseteq K_0$.
Note that the sets $K_0$, $K_0+h$, $K_1$ and $K_1+h$ are pairwise disjoint if $H\nsubseteq K$, while $K_0+h=K_0$ and $K_1+h=K_1$ if $H\subseteq K_0$.
We construct permutations $\alpha$ and $\beta$ of $V(G)$ depending on $H\nsubseteq K$ or $H\subseteq K_0$: if $H\nsubseteq K$, then define $\alpha$ and $\beta$ by
\[
x^\alpha=
\begin{cases}
x+h & \mbox{if } x\in K_0\\
x-h & \mbox{if } x\in K_0+h\\
x & \mbox{otherwise}
\end{cases}
\quad\text{and}\quad
x^\beta=
\begin{cases}
x+h & \mbox{if } x\in K_1\\
x-h & \mbox{if } x\in K_1+h\\
x & \mbox{otherwise};
\end{cases}
\]
if $H\subseteq K_0$, then define $\alpha$ and $\beta$ by
\[
x^\alpha=
\begin{cases}
x+h & \mbox{if } x\in K_0\\
x & \mbox{otherwise}
\end{cases}
\quad\text{and}\quad
x^\beta=
\begin{cases}
x+h & \mbox{if } x\in K_1\\
x & \mbox{otherwise}.
\end{cases}
\]
Clearly, $\alpha$ and $\beta$ are distinct permutations of $V(G)$.
Let $(u,v)\in A(G)$, and let $s=v-u\in S$.
Suppose that $u^\alpha=u+\delta h$ and $v^\beta=v+\varepsilon h$, where $\delta,\varepsilon\in\{-1,0,1\}$.
Then $v^\beta-u^\alpha=s+(\varepsilon-\delta)h$.
Note that $R\subseteq dZ_n=K$.
If $s\notin K_1$, then $s\notin R$, and so $H+s\subseteq S$, which implies that $v^\beta-u^\alpha\in S$.
If $s\in K_1$, then it is straightforward to verify by the definition of $\alpha$ and $\beta$ that $\varepsilon-\delta=0$, whence $v^\beta-u^\alpha=s\in S$.
Therefore, $(u,v)\in A(G)\Rightarrow(u^\alpha,v^\beta)\in A(G)$.
Let
\[
\sigma\colon\ A(G)\to A(G),\ \,(u,v)\mapsto(u^\alpha,v^\beta).
\]
Then $\sigma$ is a well-defined injection as $\alpha$ and $\beta$ are permutations.
Since $A(G)$ is finite, we deduce that $\sigma$ is a permutation on $A(G)$.
Thus, $(u^\alpha,v^\beta)\in A(G)\Leftrightarrow(u,v)\in A(G)$.
Similarly, $(u^\beta,v^\alpha)\in A(G)\Leftrightarrow(u,v)\in A(G)$.
Then it follows from Theorem~\ref{Thm-3} that $G$ is unstable.

\textsf{Case~4:} (C.4) holds.
In this case, we define permutations $\alpha$ and $\beta$ of $V(G)$ by
\[
x^\alpha=hx \quad\text{and}\quad x^\beta=hx+\frac{n}{2}.
\]
Then $\alpha\neq\beta$.
Let $(u,v)\in A(G)$, and let $s=v-u$.
Then $v^\beta-u^\alpha=(hv+n/2)-hu=hs+n/2\in S$.
Similarly, $v^\alpha-u^\beta=hv-(hu+n/2)=hs-n/2\in S$.
Then the same argument as in Case~3 yields
\[
u\to_G v\;\Leftrightarrow\;u^\alpha\to_G v^\beta\;\Leftrightarrow\;u^\beta\to_G v^\alpha,
\]
and it follows from Theorem~\ref{Thm-3} that $G$ is unstable.
\end{proof}

\section{Proof of Theorem~\ref{Thm-4}}

Let $R$ be a group, and let $S$ be a subset of $R$. Denote by $\widehat{R}$ the subgroup of $\Sym(R)$ induced by the right regular action of $R$ on itself, and denote
\[
\Aut(R,S)=\{\alpha\in\Aut(R)\mid S^\alpha=S\}.
\]
It is well known and straightforward to verify that
\begin{equation}\label{Eqn-9}
\widehat{R}\rtimes\Aut(R,S)\leq\Aut(\Cay(R,S)).
\end{equation}
If the equality in~\eqref{Eqn-9} holds, then $\Cay(R,S)$ is called \emph{normal}; otherwise, it is called \emph{nonnormal}.

\begin{lemma}\label{Lem-4}
Let $R=H\times K$ be a group with a subgroup $H$ and a characteristic subgroup $K$ such that $|K|\geq5$, and let $S=T\times(K\setminus\{1\})\subseteq H\times K$ with $T\subseteq H$. Then $\Cay(R,S)$ is nonnormal.
\end{lemma}

\begin{proof}
Suppose for a contradiction that $G=\Cay(R,S)$ is normal. View $V(G)$ as the Cartesian product of $H$ and $K$. Then the action of $\widehat{K}$ on the second coordinate gives rise to a regular subgroup of $\Sym(K)$. Since $S=T\times(K\setminus\{1\})$, the action of $\Sym(K)$ on the second coordinate induces a subgroup of $\Aut(G)$. Note that $\widehat{R}$ is normal in $\widehat{R}\rtimes\Aut(R,S)=\Aut(G)$ and that $\widehat{K}$ is characteristic in $\widehat{R}$ as $K$ is characteristic in $R$.
It follows that $\widehat{K}$ is normal in $\Sym(K)$, contradicting the fact that $\Sym(K)$ has no regular normal subgroup.
\end{proof}

The following characterization of connected arc-transitive nonnormal circulants is obtained independently by Kov\'{a}cs~\cite{Kovacs2004} and Li~\cite{Li2005} (see also~\cite[Theorem~1.1]{LXZ2021}), where $\overline{K}_d$ denotes the empty graph with $d$ vertices and $H[K]$ denotes the lexicographic product of $H$ with $K$.

\begin{proposition}\label{Prop-1}
Let $G$ be a connected arc-transitive nonnormal circulant digraph of order $n$. Then one of the following holds:
\begin{enumerate}[{\rm(a)}]
\item\label{Prop-1-a} $G=K_n$;
\item\label{Prop-1-b} $G=H[\overline{K}_d]$, where $n=md$, $d>1$ and $H$ is a connected arc-transitive circulant digraph of order $m$;
\item\label{Prop-1-c} $G=H\times K_d$, where $n=md$, $d>3$, $\gcd(m,d)=1$ and $H$ is a connected arc-transitive circulant digraph of order $m$.
\end{enumerate}
\end{proposition}

\begin{lemma}\label{Lem-7}
Let $H$ be a Cayley digraph of an abelian group. Then $H$ is twin-free if and only if it is strongly twin-free.
\end{lemma}

\begin{proof}
Let $H=\Cay(R,S)$ with abelian $R$. Then for $u,v\in R$,
\begin{align*}
H^+(u)=H^+(v)\;\Leftrightarrow\;S+u=S+v&\;\Leftrightarrow\;S+u-(u+v)=S+v-(u+v)\\
&\;\Leftrightarrow\;S-v=S-u\;\Leftrightarrow\;-S+v=-S+u\;\Leftrightarrow\;H^-(v)=H^-(u).
\end{align*}
As a consequence, $H$ is twin-free if and only if it is strongly twin-free.
\end{proof}

Recall that a digraph $G$ is said to be \emph{vertex-transitive} if $\Aut(G)$ acts transitively on $V(G)$.
In particular, Cayley digraphs are vertex-transitive.

\begin{lemma}\label{Lem-8}
Let $H$ be a strongly twin-free vertex-transitive digraph, and let $d\geq3$ be an integer coprime to the order of $H$. If $H\times K_d$ is unstable, then $H$ is unstable.
\end{lemma}

\begin{proof}
Let $G=H\times K_d$ such that $G$ is unstable. Then by Theorem~\ref{Thm-3}, there exist distinct permutations $\alpha$ and $\beta$ of $V(G)$ such that, for $u,v\in V(G)$,
\begin{equation}\label{Eqn-5}
u\to_G v\;\Leftrightarrow\;u^\alpha\to_G v^\beta\;\Leftrightarrow\;u^\beta\to_G v^\alpha.
\end{equation}
Hence, by Lemma~\ref{Lem-6}\,\eqref{Lem-6-c}, we have $\alpha,\beta\in\Aut(S(G))$. Note that $K_d$ is strongly twin-free and that $H$ has no vertex with empty out-neighborhood or in-neighborhood as $H$ is vertex-transitive and connected. We conclude from Lemma~\ref{Lem-3} that
\[
S(G)=S(H\times K_d)=S(H)\Bbox S(K_d)=S(H)\Bbox K_d.
\]
Since $H$ is vertex-transitive and each automorphism of $H$ is also an automorphism of $S(H)$, the graph $S(H)$ is vertex-transitive and hence has the form $mH_0$, the disjoint union of $m$ copies of some connected graph $H_0$, where $|V(H)|=m|V(H_0)|$. It follows that
\[
S(G)=S(H)\Bbox K_d=(mH_0)\Bbox K_d=m(H_0\Bbox K_d)
\]
with $H_0\Bbox K_d$ connected, and so
\[
\Aut(S(G))=\Aut(H_0\Bbox K_d)\wr \Sym(m).
\]
Since $\gcd(|V(H_0)|,d)=1$ as $\gcd(|V(H)|,d)=1$, we conclude from~\cite[Theorem~6.10]{HIK2011} that
\[
\Aut(H_0\Bbox K_d)=\Aut(H_0)\times\Aut(K_d).
\]
Therefore, we have $\alpha=(\alpha_1,\alpha_2)$ and $\beta=(\beta_1,\beta_2)$ for some $\alpha_1,\beta_1\in\Sym(V(H))$ and $\alpha_2,\beta_2\in\Sym(V(K_d))$. Now by~\eqref{Eqn-5}, for $u_1,v_1\in V(H)$ and $u_2,v_2\in V(K_d)$,
\begin{equation}\label{Eqn-6}
(u_1,u_2)\to_G(v_1,v_2)\;\Leftrightarrow\;(u_1,u_2)^{\alpha}\to_G(v_1,v_2)^\beta
\;\Leftrightarrow\;({u_1}^{\alpha_1},{u_2}^{\alpha_2})\to_G({v_1}^{\beta_1},{v_2}^{\beta_2}).
\end{equation}
Fixing $u_2,v_2\in V(K_d)$ such that $u_2\to_{K_d} v_2$, we deduce from~\eqref{Eqn-6} that
\[
u_1\to_H v_1\;\Rightarrow\;{u_1}^{\alpha_1}\to_H{v_1}^{\beta_1},
\]
while fixing $u_2,v_2\in V(K_d)$ such that ${u_2}^{\alpha_2}\to_{K_d}{v_2}^{\beta_2}$, we deduce from~\eqref{Eqn-6} that
\[
{u_1}^{\alpha_1}\to_H{v_1}^{\beta_1}\;\Rightarrow\; u_1\to_H v_1.
\]
This shows that $u_1\to_H v_1\,\Leftrightarrow\,{u_1}^{\alpha_1}\to_H{v_1}^{\beta_1}$, and in the same vein (swapping $\alpha$ and $\beta$) we obtain $u_1\to_H v_1\,\Leftrightarrow\,{u_1}^{\beta_1}\to_H{v_1}^{\alpha_1}$. Hence
\begin{equation}\label{Eqn-7}
u_1\to_H v_1\;\Leftrightarrow\;{u_1}^{\alpha_1}\to_H{v_1}^{\beta_1}\;\Leftrightarrow\;{u_1}^{\beta_1}\to_H{v_1}^{\alpha_1}.
\end{equation}
Similarly,
\begin{equation}\label{Eqn-8}
u_2\to_{K_d}v_2\;\Leftrightarrow\;{u_2}^{\alpha_2}\to_{K_d}{v_2}^{\beta_2}\;\Leftrightarrow\;{u_2}^{\beta_2}\to_{K_d}{v_2}^{\alpha_2}.
\end{equation}
Since $K_d$ is stable (see~\cite[Example~2.2]{QXZ2019}), we deduce from Theorem~\ref{Thm-3} and~\eqref{Eqn-8} that $\alpha_2=\beta_2$. Then the condition $\alpha\neq\beta$ forces $\alpha_1\neq\beta_1$, and it follows from Theorem~\ref{Thm-3} and~\eqref{Eqn-7} that $H$ is unstable.
\end{proof}

\begin{proof}[\bf\textup{Proof of Theorem~$\ref{Thm-4}$}]
Suppose for a contradiction that $G=\Cay(Z_n,S)$ is an arc-transitive nontrivially unstable circulant digraph of minimum order $n$.
Then $n\geq3$, and as $K_n$ is stable (see~\cite[Example~2.2]{QXZ2019}), $G$ is not a complete graph. Moreover, since $G$ is twin-free, $G$ cannot be as described in Proposition~\ref{Prop-1}\,\eqref{Prop-1-b}. Thereby we conclude from Proposition~\ref{Prop-1} that either $G$ is a normal Cayley digraph, or $G=H\times K_d$ for some connected arc-transitive circulant digraph $H$ of order $m$ such that $n=md$, $d>3$ and $\gcd(m,d)=1$. Suppose that the latter case occurs. Then since $\underline{G}$ is connected, $\underline{H}$ is connected. Moreover, since $\underline{G}$ is non-bipartite, $\underline{G}$ contains an odd cycle and so does $\underline{H}$, which implies that $\underline{H}$ is non-bipartite. If $H$ has distinct vertices $u$ and $v$ such that $H^\epsilon(u)=H^\epsilon(v)$ for each $\epsilon\in\{+,-\}$, then $G^\epsilon((u,0))=G^\epsilon((v,0))$ for each $\epsilon\in\{+,-\}$, which is not possible as $G$ is twin-free.
Hence $H$ is twin-free, and so as $H$ is a Cayley digraph of a cyclic group, it follows from Lemma~\ref{Lem-7} that $H$ is strongly twin-free. This together with Lemma~\ref{Lem-8} (note that $Z_d$ is a characteristic subgroup of $Z_n=Z_m\times Z_d$ as $\gcd(m,d)=1$) implies that $H$ is nontrivially unstable, contradicting the minimality of $G$.

Thus, $G$ is a normal Cayley digraph, which means that $\Aut(G)=\widehat{Z_n}\rtimes\Aut(Z_n,S)$. Since $G$ is arc-transitive, it follows that $\Aut(Z_n,S)$ is transitive on $S$.
Suppose that $n$ is even. Then each automorphism of $Z_n$ is induced by the multiplication by an odd integer. As $G$ is connected, there exists $s\in S$ with $s$ odd. Therefore, $S$ is the orbit of $\Aut(Z_n,S)$ containing $s$ and hence has only odd elements. This implies that $G$ is bipartite, a contradiction.

Thus, $n$ is odd. Noting $G\times K_2=\Cay(Z_n\times Z_2,S\times\{1\})$, the isomorphism $Z_n\times Z_2\cong Z_{2n}$ implies that $G\times K_2$ is a circulant digraph of order $2n$. Since $G$ and $K_2$ are both arc-transitive, $G\times K_2$ is arc-transitive. Moreover, since $G$ is twin-free, $G\times K_2$ is twin-free. Then as $G\times K_2$ is not a complete graph, we deduce from Proposition~\ref{Prop-1} that one of the following holds:
\begin{enumerate}[{\rm(i)}]
\item $G\times K_2$ is a normal Cayley digraph of $Z_n\times Z_2$;
\item $G\times K_2\cong G_1\times K_c$, where $2n=\ell c$, $c>3$, $\gcd(\ell,c)=1$ and $G_1$ is a connected arc-transitive circulant digraph of order $\ell$.
\end{enumerate}

First assume that~(i) occurs. Then by the definition of a normal Cayley digraph,
\begin{equation}\label{Eqn-10}
|\Aut(G\times K_2)|=2n|\Aut(Z_n\times Z_2,S\times\{1\})|.
\end{equation}
Since $n$ is odd, we have $\Aut(Z_n\times Z_2)=\Aut(Z_n)\times\Aut(Z_2)$ and hence
\[
\Aut(Z_n\times Z_2,S\times\{1\})=\Aut(Z_n,S)\times\Aut(Z_2)=\Aut(Z_n,S)\times1.
\]
This in conjunction with~\eqref{Eqn-10} and $\widehat{Z_n}\rtimes\Aut(Z_n,S)\leq\Aut(G)$ (see~\eqref{Eqn-9}) yields
\[
|\Aut(G\times K_2)|=2n|\Aut(Z_n,S)|=2|\widehat{Z_n}\rtimes\Aut(Z_n,S)|\leq|\Aut(G)\times\Aut(K_2)|.
\]
In view of $\Aut(G\times K_2)\geq\Aut(G)\times\Aut(K_2)$ we then conclude that
\[
\Aut(G\times K_2)=\Aut(G)\times\Aut(K_2),
\]
which means that $G$ is stable, a contradiction.

Now~(ii) occurs. By Lemma~\ref{Lem-5}, $\underline{G_1}\times K_c=\underline{G_1\times K_c}=\underline{G\times K_2}=\underline{G}\times K_2$. In particular, $\underline{G_1}\times K_c$ is bipartite and hence does not contain any odd cycle. This implies that $\underline{G_1}$ does not contain any odd cycle, and so $G_1$ is bipartite. As a consequence, $\ell$ is even.
Write $G_1=\Cay(Z_\ell,S_1)$.
Then $G_1\times K_c=\Cay(Z_\ell\times Z_c,S_1\times S_2)$ with $S_2=Z_c\setminus\{0\}$. It is straightforward to check that $(x,y)\mapsto(1-n)x+ny$ is a well-defined isomorphism from $Z_n\times Z_2$ to $Z_{2n}$. Then since $2n=\ell c$, $\gcd(\ell,c)=1$ and $c=2n/\ell$ divides $n$, we obtain an isomorphism
\[
\varphi\colon Z_n\times Z_2\rightarrow Z_\ell\times Z_c,\ (x,y)\mapsto(((1-n)x+ny)\bmod\ell,x\bmod c).
\]
This induces an isomorphism from $G\times K_2=\Cay(Z_n\times Z_2,S\times\{1\})$ to $\Cay(Z_\ell\times Z_c,(S\times\{1\})^\varphi)$, which combined with $G\times K_2\cong G_1\times K_c=\Cay(Z_\ell\times Z_c,S_1\times S_2)$ gives
\begin{equation}\label{Eqn-11}
\Cay(Z_\ell\times Z_c,S_1\times S_2)\cong\Cay(Z_\ell\times Z_c,(S\times\{1\})^\varphi).
\end{equation}
Since $\Cay(Z_\ell\times Z_c,S_1\times S_2)\cong G\times K_2$ is an arc-transitive circulant digraph,~\cite[Theorem~1.3]{LXZ2021} asserts that the digraph isomorphism~\eqref{Eqn-11} is induced by some automorphism of the group $Z_\ell\times Z_c$, meaning that $(S\times\{1\})^\varphi=(S_1\times S_2)^\sigma$ for some $\sigma\in\Aut(Z_\ell\times Z_c)$. Write $\sigma=(\sigma_1,\sigma_2)$ with $\sigma_1\in\Aut(Z_\ell)$ and $\sigma_2\in\Aut(Z_c)$. Then $S_2^{\sigma_2}=S_2$, and so
\begin{equation}\label{eq4}
(S\times\{1\})^\varphi=S_1^{\sigma_1}\times S_2.
\end{equation}
Let $T=\{t\bmod(\ell/2)\mid t\in S_1^{\sigma_1}\}\subseteq Z_{\ell/2}$. As $n=\ell c/2$ and $\gcd(\ell/2,c)=1$, we have an isomorphism
\[
\psi\colon Z_n\rightarrow Z_{\ell/2}\times Z_c,\ z\mapsto(z\bmod(\ell/2),z\bmod c).
\]
For each $s\in S$, it follows from~\eqref{eq4} that
\[
(((1-n)s+n)\bmod\ell,s\bmod c)=(s,1)^\varphi\in S_1^{\sigma_1}\times S_2,
\]
whence
\[
s^\psi=(s\bmod(\ell/2),s\bmod c)=(((1-n)s+n)\bmod(\ell/2),s\bmod c)\in T\times S_2.
\]
This implies that $S^\psi\subseteq T\times S_2$. Moreover, since $\varphi$ and $\psi$ are isomorphisms,
\[
|S^\psi|=|(S\times\{1\})^\varphi|=|S_1^{\sigma_1}\times S_2|=|S_1^{\sigma_1}||S_2|\geq|T||S_2|.
\]
Therefore, $S^\psi=T\times S_2$. As $\ell$ is even and $\gcd(\ell,c)=1$, we see that $c$ is odd and hence $c\geq5$. Then Lemma~\ref{Lem-4} shows that $G=\Cay(Z_n,S)$ is nonnormal, a contradiction. This completes the proof.
\end{proof}

\section{Proof of Theorem~\ref{Thm-5}}

\begin{lemma}\label{Lem-9}
Let $R$ be an abelian group in the additive notation, let $\alpha$ be an automorphism of a Cayley digraph $\Cay(R,S)$, and let $k$ be an integer such that $ks\neq kt$ for all distinct $s$ and $t$ in $S$. Then $\alpha$ is an automorphism of $\Cay(R,kS)$, where $kS=\{ks\mid s\in S\}$.
\end{lemma}

\begin{proof}
Write $k = p_1\cdots p_r$ with primes $p_1,\ldots,p_r$ and let $k_i = p_1\cdots p_i$ for $i\in\{0,1,\ldots,r\}$ (as usual, the convention is that $k_0=1$). We prove by induction on $i$ that $\alpha$ is an automorphism of $\Cay(R, k_iS)$. The base case $i=0$ holds trivially. Suppose that $\alpha$ is an automorphism of $\Cay(R,k_{i-1}S)$ for some $i\in\{1,\ldots,r\}$.

For $v, w \in R$, let $T_i(v, w)$ be the set of $p_i$-tuples $(x_1,\dots,x_{p_i})$ of elements in $k_{i-1}S$ such that $x_1+ \cdots + x_{p_i} = w - v$. Then $|T_i(v,w)|$ equals the number of walks of length $p_i$ from $v$ to $w$ in the digraph $\Cay(R,k_{i-1}S)$, and $|T_i(v^\alpha,w^\alpha)|=|T_i(v,w)|$ since $\alpha\in\Aut(\Cay(R,k_{i-1}S))$. Let
\[
\rho\colon(x_1, x_2, \ldots, x_{p_i})\mapsto(x_2, \ldots, x_{p_i},x_1)
\]
be a cyclic permutation for $(x_1, x_2, \dots, x_{p_i})\in T_i(v,w)$. Since $R$ is abelian, $\rho$ preserves $T_i(v,w)$ and hence partitions the set into orbits of $\rho$. Note that an orbit of $\rho$ is either of length $p_i$ or a singleton $\{(x,x,\ldots,x)\}$, where $x=k_{i-1}s\in k_{i-1}S$ with $p_ix=w-v$; moreover, if $s,t\in S$ with $p_ik_{i-1}s=p_ik_{i-1}t$, then multiplying $p_{i+1}\cdots p_r$ to both sides of the equality we obtain $ks=kt$, which forces $s=t$. Thus, either $w-v\notin p_ik_{i-1}S$ and $|T_i(v,w)|\equiv0\pmod{p_i}$, or $w-v\in p_ik_{i-1}S$ and $|T_i(v,w)|\equiv1\pmod{p_i}$. Accordingly,
\[
p_i\nmid|T_i(v,w)|\;\Leftrightarrow\;w-v\in p_ik_{i-1}S\;\Leftrightarrow\;(v,w)\in A(\Cay(R,k_iS)).
\]
This together with $|T_i(v^\alpha,w^\alpha)|=|T_i(v,w)|$ implies that $\alpha$ is an automorphism of $\Cay(R,k_iS)$, completing the proof.
\end{proof}

\begin{proof}[\bf\textup{Proof of Theorem~$\ref{Thm-5}$}]
Let $R$ be an abelian group of odd order in the additive notation, and let $G=\Cay(R,S)$ be a Cayley digraph of $R$. Then $G$ is non-bipartite. Suppose that $G$ is connected and twin-free. We aim to prove $\Aut(G\times K_2)=\Aut(G)\times Z_2$.

Take an arbitrary $\alpha\in\Aut(G\times K_2)$. Since $G$ is connected and non-bipartite, the digraph $G\times K_2$ is connected and bipartite with parts $R\times\{0\}$ and $R\times\{1\}$. Hence $\alpha$ either stabilizes or swaps the two parts, and so there exists a unique $\sigma\in1\times Z_2$ such that $\widetilde{\alpha}\coloneqq\alpha\sigma$ stabilizes both $R\times\{0\}$ and $R\times\{1\}$. Denote $k = |R| + 1$. Note that $\widetilde{\alpha}$ is an automorphism of $G\times K_2=\Cay(R\times Z_2,S\times\{1\})$.
Since $k$ is even, $k(s, 1) = (s, 0)$ for each $s\in S$. Then Lemma~\ref{Lem-9} implies that $\widetilde{\alpha}$ is an automorphism of $\Cay(R\times Z_2, S \times \{0\})$. Since $\widetilde{\alpha}$ stabilizes $R\times\{0\}$ and $R\times\{1\}$, it induces an automorphism $\overline{\alpha}$ (note that $\sigma$ is determined by $\alpha$) of $G$, the induced subgraph of $\Cay(R\times Z_2,S\times\{0\})$ on $R\times\{0\}$.

Let $\beta\in\Aut(G\times K_2)$ such that $\overline{\alpha}=\overline{\beta}$. Then $\widetilde{\beta}$ stabilizes both $R\times\{0\}$ and $R\times\{1\}$ such that $(v,0)^{\widetilde{\alpha}}=(v,0)^{\widetilde{\beta}}$ for all $v\in R$. In particular, the images of $(G\times K_2)^\epsilon((w,1))$ under $\widetilde{\alpha}$ and $\widetilde{\beta}$ are the same for all $w\in R$ and $\epsilon\in\{+,-\}$.
Since $G$ is twin-free, $G\times K_2$ is twin-free, and so the images of $(w,1)$ under $\widetilde{\alpha}$ and $\widetilde{\beta}$ are the same for all $w\in R$. This together with $\overline{\alpha}=\overline{\beta}$ means that $\widetilde{\alpha}=\widetilde{\beta}$. Hence the number of choices for $\alpha$ is at most the product of that for $\overline{\alpha}$ and that for $\sigma$; that is,
\[
|\Aut(G\times K_2)|\leq2|\Aut(G)|=|\Aut(G)\times Z_2|.
\]
Therefore, $\Aut(G\times K_2)=\Aut(G)\times Z_2$, as desired.
\end{proof}

\end{document}